\documentclass[]{imsart}
\RequirePackage[OT1]{fontenc}
\RequirePackage{amsthm,amsmath,natbib}
\RequirePackage[colorlinks,citecolor=blue,urlcolor=blue]{hyperref}
\usepackage{mathrsfs,url,color}
\usepackage{latexsym}
\usepackage{amsfonts}

\startlocaldefs
\numberwithin{equation}{section}
\theoremstyle{plain}
\newtheorem{thm}{Theorem}[section]
\newtheorem{rem}[thm]{Remark}
\newtheorem{prop}[thm]{Proposition}
\newtheorem{lem}[thm]{Lemma}

\newcommand{\thmref}[1]{Theorem~{\rm \ref{#1}}}
\newcommand{\lemref}[1]{Lemma~{\rm \ref{#1}}}
\newcommand{\propref}[1]{Proposition~{\rm \ref{#1}}}

\endlocaldefs

\begin{document}

\begin{frontmatter}

\title{On the Error Bound in the Normal Approximation for Jack Measures}
\runtitle{Normal Approximation for Jack Measures}

\begin{aug}
\author{\fnms{LOUIS H. Y.} \snm{CHEN}\thanksref{a}\ead[label=e1]{matchyl@nus.edu.sg}}
\author{\fnms{MARTIN} \snm{RAI\v{C}}\thanksref{b}
	\ead[label=e2]{martin.raic@fmf.uni-lj.si}}
\and
\author{\fnms{L\^{E} V\v{A}N} \snm{TH\`{A}NH}\thanksref{c}
\ead[label=e3]{levt@vinhuni.edu.vn}}

\address[a]{Department of Mathematics,
National University of Singapore,
    10 Lower Kent Ridge Road, 
           Singapore 119076. 
\printead{e1}}
\address[b]{University of Ljubljana, University of Primorska, and Institute of Mathematics, Physics and Mechanics, Slovenia. 
	\printead{e2}}
\address[c]{Department of Mathematics,
Vinh University,
182 Le Duan, Vinh, Nghe An,
Vietnam. 
\printead{e3}}

\runauthor{Louis H. Y. Chen et al.}

\affiliation{National University of Singapore and Vinh University}

\end{aug}

\begin{abstract}
In this paper, we obtain uniform and non-uniform bounds on the Kolmogorov distance in the normal approximation for Jack deformations of the character ratio, by using Stein's method and zero-bias couplings. Our uniform bound comes very close to that conjectured by Fulman [J. Combin. Theory Ser. A, 
108 (2004), 275--296]. As a by-product of the proof of the non-uniform bound, we obtain a Rosenthal-type inequality for zero-bias couplings.
\end{abstract}



\end{frontmatter}

\section{Introduction and main results}\label{sec:int}

 Let $G$ be a finite group, and $G^{*}$ the set of all the
irreducible representations of $G$. Then $$\sum_{\pi\in
G^{*}}{\text{dim}}(\pi)^2=|G|,$$ where ${\text{dim}}(\pi)$ denotes
the dimension of the irreducible representation $\pi$ (see 
\cite[Proposition 1.10.1]{Sagan}). The Plancherel measure is a
probability measure on $G^{*}$ defined by
\[{\mathbb{P}}(\{\pi\})=\dfrac{{\text{dim}}(\pi)^2}{|G|}.\]
Let $n$ be a positive integer. An important special case is the finite symmetric group ${\cal{S}}_n$. For this group, the
irreducible representations are parametrized by partitions
$\lambda$ of $n$,  and
the dimension of the representation associated to $\lambda$
is known to be equal to
the number of standard $\lambda$-tableaux
(see \cite[Theorem 2.6.5]{Sagan}). We also denote the number of standard
$\lambda$-tableaux by ${\text{dim}}(\lambda)$, and write a partition
$\lambda=(\lambda_1,\lambda_2,\dots,\lambda_m)$ of $n$ simply
$\lambda\vdash n$. The hooklength of a box $s$ in the partition
$\lambda$ is defined as $h(s)=a(s)+l(s)+1$.  Here $a(s)$
denotes the number of boxes in the same row of $s$ and to the right
of $s$ (the ``arm" of $s$) and $l(s)$ denotes the number of boxes in
the same column of $s$ and below $s$ (the ``leg" of $s$).  The Plancherel measure in this case
is
$${\mathbb{P}}(\{\lambda\})=\dfrac{{\text{dim}}(\lambda)^2}{n!}.$$
By the hook formula (see, e.g., \cite{Sagan}) which
states that
$${\text{dim}}(\lambda)=\dfrac{n!}{\Pi_{s\in \lambda}h(s)},$$
where the product is over boxes in the partition and $h(s)$ is the
hooklength of a box $s$, we also have
\begin{equation}\label{plan11}{\mathbb{P}}(\{\lambda\})=\dfrac{n!}{\Pi_{s\in \lambda}h^2(s)}.\end{equation}

A random partition $\lambda$ chosen by the Plancherel measure
has interesting connections to the Gaussian unitary ensemble (GUE) of random matrix theory.
We recall that the joint probability density of the eigenvalues $x_1\ge x_2\ge\dots\ge x_n$ of the  Gaussian orthogonal ensemble (GOE), Gaussian unitary ensemble (GUE), and Gaussian symplectic ensemble (GSE)   is given by
\begin{equation}\label{rm11}
\dfrac{1}{Z_{\beta}}\exp\left(-\dfrac{x_{1}^2+\dots+x_{n}^2}{2}\right)\Pi_{1\le i<j\le n} (x_i-x_j)^{\beta}
\end{equation}
with $\beta=1,2,4$, respectively. Here $Z_{\beta}$ is a normalization constant.
Let $\pi$ be a permutation chosen from the uniform measure
of the symmetric group ${\cal{S}}_n$ and $l(\pi)$ the length of the
longest increasing subsequence in $\pi$.
\cite{BDJ} proved that $(l(\pi)-2\sqrt{n})/n^{1/6}$ converges to the Tracy-Widom distribution as $n \rightarrow \infty$.
It follows from the Robinson-Schensted-Knuth correspondence (see 
\cite{Sagan}) that the first row of a random partition distributed according to the
Plancherel measure has the same distribution as the longest increasing subsequence of a
random permutation distributed according to the uniform measure.
So the result of 
\cite{BDJ} says that a suitably normalized length of the first row of a random partition distributed according to
the Plancherel measure converges to the Tracy-Widom distribution.
 \cite{BOO},  \cite{Johansson}
proved that the joint distribution of suitably normalized lengths of the rows of a random partition distributed according to the Plancherel measure converges to the joint distribution of the eigenvalues $x_1\ge x_2\ge\dots\ge x_n$ of a $n\times n$ GUE matrix.

Jack$_{\alpha}$ measure is an extension of the Plancherel measure. For $\alpha>0$, the Jack$_{\alpha}$ measure is a probability measure
on the set of all partitions of a positive integer $n$, which chooses a partition
$\lambda$ with probability
$${\mathbb{P}}_{\alpha}\left(\{\lambda\}\right)=\dfrac{\alpha^n n!}
{\Pi_{s\in \lambda}(\alpha a(s)+l(s)+1)(\alpha a(s)+l(s)+\alpha)},$$
where the product is over all boxes in the partition. For
example, the partition
\begin{equation*}
\lambda=
\begin{array}{lrc}
\Box\  \ \Box\  \ \Box\\
\Box\ \ \Box\\
\Box
\end{array}
\end{equation*}
 of $6$ has Jack$_\alpha$ measure
 $$\dfrac{720\alpha^3}{(3\alpha+2)(2\alpha+3)(\alpha+2)^2(2\alpha+1)^2}.$$

We notice that the Jack measure with parameter $\alpha=1$ agrees the
Plancherel measure of the symmetric group since it coincides with \eqref{plan11}.
It is mentioned in \cite{Matsumoto08} that for any positive real number
$\beta> 0$, the Jack$_\alpha$ measure with $\alpha=2/\beta$ is the counterpart of the Gaussian $\beta$-ensemble
$(G\beta E)$ with the probability density function proportional to \eqref{rm11}.

Let $\lambda$ be a partition of $n$ chosen from the Plancherel
measure of the symmetric group ${\cal{S}}_n$, and
$\chi^{\lambda}(12)$ the character of the irreducible representation associated to
$\lambda$ evaluated on the transposition $(12)$.
Characters of the irreducible representations of a symmetric group are of interest in the literature because 
they play central roles in representation theory and other fields of mathematics such as random walks (\cite{DS})
and the moduli space of curves (\cite{EskinOkounkov}).  The quantity
$\chi^{\lambda}(12)/{\text{dim}}(\lambda)$, which is a normalization of $\chi^{\lambda}(12)$, is called a character ratio. As $\lambda$ is distributed according to the Plancherel measure, $\chi^{\lambda}(12)$ is a random variable.

In \cite{Kerov}, it is stated that
\begin{equation}\label{cr11}
\dfrac{\sqrt{{n
\choose 2}}\chi^{\lambda}(12)}{{\text{dim}}(\lambda)}
\end{equation}
is asymptotically normal with mean $0$ and variance $1$ as $n \rightarrow \infty$. A 
proof of Kerov's central limit theorem can be found in \cite{Hora}, which uses the
method of moments and combinatorics. More recently, a proof in \cite{Sniady} uses the genus expansion of
random matrix theory, and another in \cite{HoraObata} 
uses quantum probability.

By a formula due to \cite{Fro00} (see also \cite{Fulman06b}), we have
\begin{equation}\label{cr12}
\dfrac{\chi^{\lambda}(12)}{{\text{dim}}(\lambda)}=
\dfrac{1}{{n \choose 2}}\sum_{i}\left({{\lambda_i}\choose
2}-{{\lambda_{i}^{'}} \choose 2}\right).
\end{equation}
Now, for $\alpha>0$, the random variable we will study
in this paper is
\begin{equation}\label{J00}W_{n,\alpha}=W_{n,\alpha}(\lambda)=\dfrac{\sum_{i}\left(\alpha{{\lambda_i}
\choose 2}-{{\lambda_{i}^{'}} \choose 2}\right)}{\sqrt{\alpha{n
\choose 2}}},\end{equation} where $\lambda$ is chosen from the Jack$_\alpha$
measure on partitions of a positive integer $n$, $\lambda_i$ is the length of the
$i$-th row of $\lambda$ and $\lambda_{i}^{'}$ is the length of the
$i$-th column of $\lambda$. By \eqref{cr12}, $W_{n,\alpha}$ coincides with \eqref{cr11} when $\alpha=1$.
 Therefore, the value $W_{n,\alpha}$ is regarded as a Jack deformation of the character ratio. 
Moreover, as remarked by \cite{Fulman04}, when $\alpha=2$,
 $W_{n,2}$ is the value of a spherical function corresponding to the Gelfand pair $(S_{2n},H_{2n})$, where $H_{2n}$ is the hyperoctahedral group of size $2^nn!$.

Normally approximation for $W_{n,\alpha}$ has been studied by \cite{Fulman04,
Fulman06b}, \cite{ShaoSu}, and 
\cite{FulmanGoldstein} by using Stein's method (see, e.g., \cite{Stein86}).
In \cite{Fulman04}, the author proved that for any fixed $\alpha\ge 1$,
\begin{equation}\label{J02}\sup_{x\in {\mathbb{R}}}\left|{\mathbb{P}}_{\alpha}(W_{n,\alpha}\le
x)-\Phi(x)\right|\le \dfrac{C_\alpha}{n^{1/4}},\end{equation}
where $C_\alpha$ is a constant depending only on $\alpha$, $\Phi(x)=\dfrac{1}{\sqrt{2\pi}}\int_{-\infty}^{x}\exp(-t^2/2)\mathrm{d}t$ is
the distribution function of the standard normal distribution.

The bound  $C_\alpha n^{-1/4}$ was later improved in \cite{Fulman06b} to
$C_\alpha n^{-1/2}$ using an inductive approach to Stein's method.  We note that in all these results, $\alpha>0$ is fixed, but we do not know how $C_\alpha$ depends on $\alpha$. An explicit constant is obtained by
\cite{ShaoSu} only when $\alpha=1$. More precisely, when $\alpha=1$, 
\cite{ShaoSu} obtained the rate $761n^{-1/2}$ by using Stein's
method for exchangeable pairs. More recently, \cite{DolegaF}
proved the Berry-Esseen bound for the multivariate case with rate $C_\alpha n^{-1/4}$, and
\cite{DolegaS} proved a general multivariate central limit theorem
for the case where $\alpha=\alpha(n)$ varying with $n$, 
satisfying
\[\dfrac{-\sqrt{\alpha}+1/\sqrt{\alpha}}{\sqrt{n}}=g_1+\dfrac{g_2}{\sqrt{n}}+o\left(\dfrac{1}{\sqrt{n}}\right),\]
where $g_1$ and $g_2$ are constants.

\cite{Fulman04} conjectured that for general $\alpha\ge 1$, the correct bound is a universal constant multiplied by $\max\left\{\dfrac{1}{\sqrt{n}},\dfrac{\sqrt{\alpha}}{n}\right\}$. 
While this bound was conjectured for the Kolmogorov distance in \eqref{J02}, using Stein's method and zero-bias couplings, \cite{FulmanGoldstein} proved that it is indeed the correct bound for the Wasserstein distance for $W_{n,\alpha}$. By the result in \cite{FulmanGoldstein}, the central limit theorem for $W_{n,\alpha}$
holds for $\alpha=\alpha(n)$ varying with $n$ as long as $\sqrt{\alpha}/n\to 0$. 
As observed by \cite{Fulman04}, this is necessary for $W_{n,\alpha}$ to be asymptotically normal.
The bound conjectured by \cite{Fulman04} for the Kolmogorov distance remains unsolved as bounds on the Kolmogorov distance are usually harder to obtain than bounds on the Wasserstein
distance. This paper is an attempt to prove the conjecture of \cite{Fulman04} for the Kolmogorov distance. We use Stein's method and zero-bias couplings to obtain both uniform and non-uniform error bounds on the Kolmogorov distance for $W_{n,\alpha}$. We have obtained a uniform error bound which comes very close to that conjectured by \cite{Fulman04}. Besides, we have obtained a very small constant. As a by-product of the proof of the non-uniform bound, we obtain a Rosenthal-type inequality for zero-bias couplings.

Throughout this paper, $Z$ denotes the standard normal random variable and
$\Phi(x)=\dfrac{1}{\sqrt{2\pi}}\int_{-\infty}^{x}\exp(-t^2/2)\mathrm{d}t$ its distribution function.
For a positive number $x$,  $\log x$ denotes the natural logarithm of $x$, and 
$\lfloor x \rfloor$ denotes the greatest integer number that is less than or equal to $x$.
For a set $S$, the indicator
function of $S$ is denoted by ${\textbf{1}}(S)$ and the cardinality of $S$
denoted by $|S|$.
For $p\ge 1$ and a random variable $X$, $\left(\mathbb{E}|X|^p\right)^{1/p}$ is denoted by $\|X\|_p$. The symbol $C_{p}$ denotes a generic positive
constant 
which is bounded by $B^p$ for some universal constant $B$,
but can be different for
each appearance. We denote Jack$_\alpha$ measure by ${\mathbb{P}}_{\alpha}$.

\begin{thm}\label{Ja1} Let $n\ge 3$ be an integer. Let $\alpha>0$ and $W_{n,\alpha}$  be as in \eqref{J00}. Then
\begin{equation*}
\sup_{x\in {\mathbb{R}}}|{\mathbb{P}}_{\alpha}(W_{n,\alpha}\le
x)-\Phi(x)|\le 8.2\max\left\{\dfrac{1}{\sqrt{n}},\dfrac{\max\left\{\sqrt{\alpha},1/\sqrt{\alpha}\right\}\log n}{n}\right\}.
\end{equation*}
\end{thm}

\begin{rem}\label{rm1} {\rm If $\dfrac{\log^2n}{n}\le \alpha\le \dfrac{n}{\log^2n}$, then the bound in \thmref{Ja1} is $\dfrac{8.2}{\sqrt{n}}$.  
For $\alpha\ge 1$, the bound in \thmref{Ja1} is $8.2\max\left\{\dfrac{1}{\sqrt{n}},\dfrac{\sqrt{\alpha}\log n}{n}\right\}$, 
which is very close to that conjectured by \cite{Fulman04}}. 
\end{rem}

We prove \thmref{Ja1} by using Stein's method for zero bias couplings. Non-uniform bounds on the Kolmogorov distance
 in the normal approximation for independent random variables using Stein's method were first investigated by \cite{ChenShao01}. 
 Stein's method has also been used to study non-uniform  bounds on the Kolmogorov distance (\cite{ChenShao04}) and concentration inequalities 
 (\cite{CD10}) for dependent random
variables.
The method developed in this paper also allows us to obtain a non-uniform bound on the Kolmogorov distance, which we state in the following theorem.

\begin{thm}\label{Ja2} Let $n\ge 3$ be an integer. Let $p\ge 2$, $1/n^2<\alpha<n^2$ and $W_{n,\alpha}$ be as in \eqref{J00}. Then for all $x\in {\mathbb{R}}$, we have
\begin{equation*}
|{\mathbb{P}}_{\alpha}(W_{n,\alpha}\le
x)-\Phi(x)|\le \dfrac{C_p}{1+|x|^p}\left(\dfrac{p^2}{\log p}\right)^p\max\left\{\dfrac{1}{\sqrt{n}},\dfrac{\max\{\sqrt{\alpha},1/\sqrt{\alpha}\}\log n}{n}\right\}.
\end{equation*}
\end{thm}
\begin{rem} {\rm
If $\alpha\ge n^2$ and $p\ge 2$, then it will be shown in the appendix that 
\begin{equation}\label{rem_add_01}
{\mathbb{E}}\left| W_{n,\alpha}\right|^p \le C_p\left(\dfrac{p^2}{\log p}\right)^p\left(\dfrac{\sqrt{\alpha}}{n}\right)^{p-2}.
\end{equation}
Therefore, by applying Markov's inequality, $|{\mathbb{P}}_{\alpha}(W_{n,\alpha}\le
x)-\Phi(x)|$ is bounded by 
\[\dfrac{C_p}{1+|x|^p}\left(\dfrac{p^2}{\log p}\right)^p\left(\dfrac{\sqrt{\alpha}}{n}\right)^{p-2}.\]
}
\end{rem}

\section{A Rosenthal-type inequality for zero-bias couplings}\label{sec:Rosenthal}

It was shown in \cite{GoldsteinReinert97} that
for any mean zero random variable $W$ with positive finite variance
$\sigma^2$, there exists a random variable $W^{*}$ which satisfies
\begin{equation}\label{zer}
{\mathbb{E}}Wf(W)=\sigma^2{\mathbb{E}}f^{'}(W^{*})
\end{equation} 
for all absolutely
continuous $f$ with ${\mathbb{E}}|Wf(W)|<\infty$.  The random variable $W^{*}$ and its distribution are called  $W$-zero biased. 
\cite{GoldsteinReinert97} (see also in 
\cite[Proposition 2.1]{CGS}) showed that the distribution of $W^{*}$
is absolutely continuous with the density
$g(x)={\mathbb{E}}[W{\textbf{1}}(W > x)]/\sigma^2.$

In this section, we prove a
Rosenthal-type inequality for zero-bias couplings, which we state as a proposition below. We will show later
that this proposition can be applied to obtain the Rosenthal inequality for sums of independent random variables. 
The use of a Rosenthal-type inequality is crucial for obtaining a non-uniform bound on the Kolmogorov distance.

\begin{prop}\label{ros}
 Let $W$ be a random variable with mean zero and variance $\sigma^2 > 0$ and let $W^{*}$ be 
$W$-zero biased. Assume that $W$ and $W^*$ are defined on the same probability space.
Let $T =W^{*}-W$. Then for every $p\ge 2$,
\begin{equation}\label{zer1}{\mathbb{E}}|W|^p\le
\kappa_p \left(\sigma^p+\sigma^2{\mathbb{E}}|T|^{p-2}\right),
\end{equation}
where 
\[\kappa_p=\dfrac{\left(\log 8\right)^3}{196}\left(\dfrac{7p}{4\log p}\right)^p.\]
\end{prop}

\begin{proof}
Let 
\begin{equation}  \label{f-R}
f(x)=
\begin{cases}
x^{p-1}&\text{ if  } x\ge0,\\
-(-x)^{p-1} &\text{ if  }x<0.\\
\end{cases}
\end{equation}
Then $f^{'}(x)=(p-1)|x|^{p-2}$ and $xf(x)=|x|^p$.

If $2\le p\le 4$, then 
\begin{equation}\label{zer21}
\begin{split} {\mathbb{E}}|W|^p&={\mathbb{E}}Wf(W)=\sigma^2
{\mathbb{E}}f^{'}(W+T )\\
&=\sigma^2(p-1){\mathbb{E}}|W+T |^{p-2}\\
&\le \sigma^2(p-1)\max\{1,2^{p-3}\}\left({\mathbb{E}}|W|^{p-2}+{\mathbb{E}}|T|^{p-2}\right)\\
&\le \sigma^2(p-1)\max\{1,2^{p-3}\}\left(\sigma^{p-2}+{\mathbb{E}}|T|^{p-2}\right)\\
&=(p-1)\max\{1,2^{p-3}\}\left(\sigma^p+\sigma^2{\mathbb{E}}|T|^{p-2}\right). 
\end{split}
\end{equation} 
Elementary calculus shows that 
\[(p-1)\max\{1,2^{p-3}\} \le \dfrac{\left(\log 8\right)^3}{196}\left(\dfrac{7p}{4\log p}\right)^p\]
for all $2\le p\le 4$. Therefore, from
\eqref{zer21}, we see that \eqref{zer1} holds for
all $2\le p \le 4$.

If $p>4$, by
Jensen's inequality, we have for all $0<\theta<1$,
\begin{equation}\label{zer23e}
\begin{split} 
{\mathbb{E}}|W|^p&={\mathbb{E}}Wf(W)=\sigma^2
{\mathbb{E}}f^{'}(W+T )\\
&=\sigma^2(p-1){\mathbb{E}}|W+T |^{p-2}\\
&\le  \sigma^2(p-1)\left(\theta {\mathbb{E}}\left(\dfrac{|W|}{\theta}\right)^{p-2}+ (1-\theta)
{\mathbb{E}}\left(\dfrac{|T|}{1-\theta}\right)^{p-2}\right)\\
&=  \sigma^2(p-1)\left(\dfrac{{\mathbb{E}}|W|^{p-2}}{\theta^{p-3}}+ \dfrac{{\mathbb{E}}|T|^{p-2}}{(1-\theta)^{p-3}}\right).
\end{split}
\end{equation} 
By using following inequality 
\begin{equation}\label{a10}
x^{\alpha}y^{1-\alpha}\le x+y \text{ for all
}0<\alpha<1,x\ge0,y\ge0,
\end{equation}
we have
\begin{equation}\label{zer23b}
\begin{split} 
{\mathbb{E}}\left(\sigma^2 |T|^{p-4}\right)&={\mathbb{E}}\left((\sigma^{p-2})^{\frac{2}{p-2}}(|T|^{p-2})^{\frac{p-4}{p-2}}\right)\\
&\le {\mathbb{E}}\left(\sigma^{p-2}+|T|^{p-2}\right)\\
&= \sigma^{p-2}+{\mathbb{E}}|T|^{p-2}.
\end{split}
\end{equation}

For the case where $4< p\le 6$, 
\eqref{zer21} and \eqref{zer23b} yield
\begin{equation}\label{zer23c}
\begin{split} 
{\mathbb{E}}|W|^{p-2}&
\le (p-3)\max\{1,2^{p-5}\}\left(\sigma^{p-2}+\sigma^2{\mathbb{E}}|T|^{p-4}\right)\\
&\le (p-3)\max\{1,2^{p-5}\}\left(2\sigma^{p-2}+{\mathbb{E}}|T|^{p-2}\right).
\end{split}
\end{equation}
By letting $\theta:=\theta_1=1/2$, we have from \eqref{zer23e} that
\begin{equation}\label{zer23}
\begin{split} 
{\mathbb{E}}|W|^{p}&\le \sigma^2(p-1)2^{p-3}\left({\mathbb{E}}|W|^{p-2}+{\mathbb{E}}|T|^{p-2}\right).
\end{split}
\end{equation} 
Combining \eqref{zer23} and \eqref{zer23c}, we obtain
\begin{equation}\label{zer23d}
\begin{split}  
{\mathbb{E}}|W|^p &\le (p-1)(p-3)2^{p-2}\max\{1,2^{p-5}\}\left(\sigma^{p}+\sigma^2{\mathbb{E}}|T|^{p-2}\right).
\end{split}
\end{equation} 
Numerical calculations show that 
\[(p-1)(p-3)2^{p-2}\max\{1,2^{p-5}\} \le \dfrac{\left(\log 8\right)^3}{196}\left(\dfrac{7p}{4\log p}\right)^p\]
for all $4< p\le 6$. Therefore, from
\eqref{zer23d}, we see that \eqref{zer1} holds in this case.

For the case where $6< p\le 8$, \eqref{zer23b} and \eqref{zer23d} yield
\begin{equation}\label{zer23g}
\begin{split} 
{\mathbb{E}}|W|^{p-2}&
\le (p-3)(p-5)2^{p-4}\max\{1,2^{p-7}\}\left(2\sigma^{p-2}+{\mathbb{E}}|T|^{p-2}\right).
\end{split}
\end{equation}
By letting $\theta:=\theta_2=2/3$, we have from \eqref{zer23e} that
\begin{equation}\label{zer23f}
\begin{split} 
{\mathbb{E}}|W|^{p}&\le \sigma^2(p-1)\left(\dfrac{3}{2}\right)^{p-3}\left({\mathbb{E}}|W|^{p-2}+2^{p-3}{\mathbb{E}}|T|^{p-2}\right).
\end{split}
\end{equation} 
Combining \eqref{zer23g} and \eqref{zer23f}, we obtain
\begin{equation}\label{zer23h}
\begin{split}  
{\mathbb{E}}|W|^p &\le (p-1)(p-3)(p-5)3^{p-3}\max\{1,2^{p-7}\}\left(\sigma^{p}+\sigma^2{\mathbb{E}}|T|^{p-2}\right).
\end{split}
\end{equation} 
Numerical calculations also show that 
\[(p-1)(p-3)(p-5)3^{p-3}\max\{1,2^{p-7}\} \le \dfrac{\left(\log 8\right)^3}{196}\left(\dfrac{7p}{4\log p}\right)^p\]
for $6< p\le 8$. Therefore, from
\eqref{zer23h}, we see that \eqref{zer1} holds in this case.

For the case where $p>8$, we prove the result by induction.  Assume that \eqref{zer1} holds for $p-2$.
By
induction and \eqref{zer23b}, we have 
\begin{equation}\label{zer24}
\begin{split} 
{\mathbb{E}}|W|^{p-2}
&\le \kappa_{p-2}\left(\sigma^{p-2}+\sigma^2{\mathbb{E}}|T|^{p-4}\right)\\
&\le \kappa_{p-2}\left(2\sigma^{p-2}+{\mathbb{E}}|T|^{p-2}\right).
\end{split}
\end{equation}
Combining \eqref{zer23e} and \eqref{zer24},
we obtain
\begin{equation}\label{zer25}
\begin{split}  
{\mathbb{E}}|W|^p
&\le  (p-1)\left(\dfrac{2\kappa_{p-2} }{\theta^{p-3}}\sigma^{p}+\left(\dfrac{\kappa_{p-2} }{\theta^{p-3}}+ \dfrac{1}{(1-\theta)^{p-3}}\right)\sigma^2{\mathbb{E}}|T|^{p-2}\right).
\end{split}
\end{equation}
The proof is completed if we can choose $0<\theta<1$ such that
\begin{equation}\label{a06}
\begin{split}  
\dfrac{2(p-1)\kappa_{p-2} }{\theta^{p-3}} \le \kappa_p \ \text{ and }\ \dfrac{p-1}{(1-\theta)^{p-3}}\le \dfrac{\kappa_{p} }{2}.
\end{split}
\end{equation}
By \lemref{lem_ap02} in the Appendix, we have
	\begin{equation}\label{a05}
\begin{split}
\kappa_p
&\ge 8\left(\dfrac{p-1}{\log (p-1)}\right)^2\kappa_{p-2}.
\end{split}
\end{equation}
Let
\[\theta=\theta(p):=\left(\dfrac{\log^2(p-1)}{4(p-1)}\right)^{1/(p-3)}.\]
Then $0<\theta<1$ and the first half of \eqref{a06} holds by \eqref{a05}.
By \lemref{lem_ap01} (in the Appendix), the second half of \eqref{a06} holds.

The proof of the proposition is completed.
\end{proof}

We now present a simple proof of
the Rosenthal inequality (\cite{Rosenthal}) for sums of mean zero independent random variables by
using \propref{ros}. 	If $\{X_i,1\le i\le n\}$ are independent symmetric random variables, \cite{JSZ} proved that 
\begin{equation}\label{ro1x}
\left\|\sum_{i=1}^nX_i\right\|_p\le \dfrac{Kp}{\log p} \max\left\{\left\|\sum_{i=1}^nX_i\right\|_2,\left(\sum_{i=1}^n \|X_i\|_{p}^{p}\right)^{1/p}\right\}\ 
\text{ for all } p\ge 2,
\end{equation}
where $K$ is a universal constant satisfying $\dfrac{1}{e\sqrt{2}}\le K \le 7.35$.
\cite{JSZ} also proved that the rate $p/\log p$ is optimal.
\cite{Latala} showed that \eqref{ro1x} holds with $K$ approximately equal to $2e$ (see Theorem 2 and Corrolary 3 in \cite{Latala}).
In \cite{Ibragimov}, the authors
proved that the constant $K$ in
\eqref{ro1x} is approximated $1/e$ when $p$ large enough (see the Corrolary in page 259 in \cite{Ibragimov}).
However, we are not aware of any result in the literature (even with assuming the symmetry
of the random variables) which proved
\eqref{ro1x} holds
with $K\le 3.5$ for all $p\ge 2$  as given in the following proposition.

\begin{prop}\label{appl}
Let $p\ge 2$ and $\{X_i, 1\le i\le n\}$ be a collection of $n$
independent mean zero random variables with $\mathbb{E}|X_i|^p<\infty$, $1\le i\le n$. 
Then
\begin{equation}\label{ro1a}
\left\|\sum_{i=1}^nX_i\right\|_p\le \dfrac{3.5p}{\log p} \max\left\{\left\|\sum_{i=1}^nX_i\right\|_2,\left(\sum_{i=1}^n \|X_i\|_{p}^{p}\right)^{1/p}\right\}.
\end{equation}
\end{prop}

\begin{proof}
Let $W=\sum_{i=1}^nX_i $ and $\sigma^2 = \mathrm{Var}(W)$. Denote $\mathrm{Var}(X_i)$ by $\sigma_i^2,~ 1\le i\le n$. Let $X_{i}^{*}$ have
the $X_i$-zero biased distribution with $\{X_{i}^{*}, 1\le i\le n\}$
mutually independent and $X_i^*$ independent of $\{X_j,j\not=i\}$. Let $I$ be a random index, independent of
$\{X_i,X_{i}^{*},1\le i\le n\}$, with the distribution
$${\mathbb{P}}(I=i)=\dfrac{\sigma_{i}^2}{\sigma^2}.$$
The argument proving part (v) of Lemma 2.1 in 
\cite{GoldsteinReinert97} shows that removing $X_I$ and replacing it
by $X_{I}^{*}$ gives a random variable $W^{*}$ with the $W$-zero
biased distribution, that is,
$$W^{*}=W-X_I+X_{I}^{*}$$
has the $W$-zero biased distribution. 

Let $\kappa_p$ be as in \propref{ros}. By \propref{ros}, we have
\begin{equation}\label{ro11}\begin{split} {\mathbb{E}}|W|^p&\le \kappa_p(\sigma^p+\sigma^2{\mathbb{E}}|W^{*}-W|^{p-2})\\
&=\kappa_p(\sigma^p+\sigma^2{\mathbb{E}}|X_I-X_{I}^{*}|^{p-2})\\
& =\kappa_p\left(\sigma^p+\sigma^2\sum_{i=1}^n{\mathbb{E}}|X_i-X_{i}^{*}|^{p-2}\sigma_{i}^2/\sigma^2\right)\\
& \le
\kappa_p\left(\sigma^p+\max\{1,2^{p-3}\}\sum_{i=1}^n\sigma_{i}^2({\mathbb{E}}|X_i|^{p-2}+{\mathbb{E}}|X_{i}^{*}|^{p-2})\right)\\
& \le
\kappa_p\left(\sigma^p+2^{p-2}\sum_{i=1}^n\sigma_{i}^2({\mathbb{E}}|X_i|^{p-2}+{\mathbb{E}}|X_{i}^{*}|^{p-2})\right).
\end{split} 
\end{equation} 
By H\"{o}lder's inequality, we have for all $1\le i\le n$,
\begin{equation}\label{ro12}\sigma_{i}^2{\mathbb{E}}|X_i|^{p-2}\le
({\mathbb{E}}|X_i|^p)^{2/p}({\mathbb{E}}|X_i|^p)^{(p-2)/p}={\mathbb{E}}|X_i|^p.\end{equation} With the function
$f$ as defined in (\ref{f-R}), it follows from \eqref{zer} that
\begin{equation}\label{ro13}(p-1)\sigma_{i}^2{\mathbb{E}}|X_{i}^{*}|^{p-2}={\mathbb{E}}|X_i|^p.\end{equation}
Combining \eqref{ro11}-\eqref{ro13}, we have $${\mathbb{E}}|W|^p\le
\kappa_p\left(\sigma^p+2^{p-1}\sum_{i=1}^n{\mathbb{E}}|X_i|^{p}\right)\le 2^p\kappa_{p} \max\left\{\sigma^p,\sum_{i=1}^n {\mathbb{E}}|X_i|^p\right\},$$
which proves \eqref{ro1a}.
\end{proof}

\section{Uniform and non-uniform Kolmogorov bounds for zero-bias couplings}\label{sec:main}

Optimal bounds on the Kolmogorov distance for zero-bias couplings have already been obtained by
\cite{Goldstein05} provided the difference between the original random variable and its
zero bias transform is properly bounded. In this section, we improved the mentioned result by 
\cite{Goldstein05} in two directions: firstly, a truncation argument is used to go beyond boundedness, and secondly, 
non-uniform bounds with polynomial decay are provided. The following
theorem gives the Kolmogorov bound in normal approximation for $W^{*}$.

 \begin{thm}\label{star} Let $W$ be such that ${\mathbb{E}}W = 0$ and $\mathrm{Var}(W) = 1$, 
 and let $W^{*}$ be
$W$-zero biased and be defined on the same probability space as
$W$. Let $T=W^{*}-W$.

(i) We have
 \begin{equation}\label{bo01}\sup_{x\in\mathbb{R}}|{\mathbb{P}}(W^{*}\le x)-\Phi(x)|~\le \left(1+\dfrac{\sqrt{2\pi}}{4}\right)\sqrt{{\mathbb{E}}T^2}.\end{equation}
 
 (ii) Let $p\ge 2$. Then for all $x \in \mathbb{R}$,
\begin{equation}\label{nou01} 
\begin{split}
\left|{\mathbb{P}}(W^* \le x) - \Phi(x)\right|
& \le \dfrac{C_p}{1+|x|^p}\left(\dfrac{p}{\log p}\right)^p\left(\sqrt{{\mathbb{E}}T^2}+\sqrt{{\mathbb{E}}|T| ^{2p+2}}\right).
\end{split}
\end{equation}
 \end{thm}
\begin{proof}
For $x\in {\mathbb{R}}$,
let $f_{x}$ be
the unique bounded solution of the Stein equation
\begin{equation}\label{bo02b}f^{'}(w)-w f(w)=1(w\le x)-\Phi(x),\end{equation}
and let
\begin{equation}\label{bo02c}
g_x(w)=(w f_x(w))^{'}.
\end{equation}
We have $0<f_x(w)\le \sqrt{2\pi}/4$ and $|f_{x}^{'}(w)|\le 1$ for all $w\in \mathbb{R}$ (see \cite{Stein86}). Therefore
\begin{equation}\label{nu01}
|g_x(w)|=|f_x(w)+wf_{x}^{'}(w)|\le 1+|w| \text{ for all }w\in \mathbb{R},
\end{equation}
\begin{equation}\label{bo03}
{\mathbb{E}}|Tf_x(W+T)|\le \dfrac{\sqrt{2\pi}}{4}{\mathbb{E}}|T| \le \dfrac{\sqrt{2\pi}}{4}\sqrt{{\mathbb{E}}T^2},
\end{equation}
and
\begin{equation}\label{bo04}
{\mathbb{E}}|(W(f_x(W+T)-f_x(W))|\le {\mathbb{E}}|WT|\le \sqrt{{\mathbb{E}}W^2{\mathbb{E}}T^2}=\sqrt{{\mathbb{E}}T^2}.
\end{equation}
Since
\begin{equation}\label{bo05}
\begin{split}
|{\mathbb{P}}(W^{*}\le x)-\Phi(x)|
& =|{\mathbb{E}}f_{x}^{'}(W^{*})-{\mathbb{E}}W^{*}f_x(W^{*})|\\
& =|{\mathbb{E}}Wf_x(W)-{\mathbb{E}}(W+T)f_x(W+T)|\\
&\le {\mathbb{E}}|(W(f_x(W+T)-f_x(W))|+{\mathbb{E}}|Tf_x(W+T)|,
\end{split}
\end{equation}
the conclusion \eqref{bo01} follows by combining \eqref{bo03}, \eqref{bo04}, and \eqref{bo05}. \thmref{star}(i) is proved.

To prove \thmref{star}(ii), it suffices to consider the case where $x\ge0$ since we can simply apply the result to $-W^{*}$ when $x<0$ (see (2.59) in 
\cite{CGS}).
In view of the uniform bound \eqref{bo01}, it suffices to consider the case where
$x\ge2$.
By applying Markov's inequality and \propref{ros}, we have
\begin{equation}\label{a32_ed01}
\begin{split}
|P(W^*\le x)-\Phi(x)| &\le \max \left\{P(W^*>x),1-\Phi(x)\right\}\\
& \le \max \left\{\dfrac{{\mathbb{E}}|W^*|^{p+1}}{x^{p+1}},1-\Phi(x)\right\}\\
& = \max \left\{\dfrac{{\mathbb{E}}|W|^{p+3}}{(p+2)x^{p+1}},1-\Phi(x)\right\}\\
& \le \max \left\{\dfrac{\kappa_{p+3}\left(1+{\mathbb{E}}|T|^{p+1}\right)}{(p+2)x^{p+1}},1-\Phi(x)\right\}.
\end{split}
\end{equation}
By using the fact that $\sqrt{2\pi}(1-\Phi(x))\le e^{-x^2/2}/x$ for all $x>0$, we have
\begin{equation}\label{a32_ed02}
\begin{split}
\max_{x>0}x^{p+1}(1-\Phi(x))\le \dfrac{1}{\sqrt{2\pi}}\max_{x>0}x^{p}e^{-x^2/2}=\dfrac{1}{\sqrt{2\pi}}\left(\dfrac{\sqrt{p}}{\sqrt{e}}\right)^p.
\end{split}
\end{equation}
Combining \eqref{a32_ed01} and \eqref{a32_ed02}, we obtain
\begin{equation}\label{a32}
\begin{split}
|P(W^*\le x)-\Phi(x)| &\le \dfrac{C_p}{1+x^p}\left(\dfrac{p}{\log p}\right)^p \left(1+\sqrt{{\mathbb{E}}|T|^{2p+2}}\right).
\end{split}
\end{equation}
If ${\mathbb{E}}|T|^{2p+2}\ge 1$, then $1+\sqrt{{\mathbb{E}}|T|^{2p+2}}\le 2\sqrt{{\mathbb{E}}|T|^{2p+2}}$.
Therefore \eqref{nou01} holds by \eqref{a32}. It 
remains to consider the case where ${\mathbb{E}}|T|^{2p+2}< 1$. In this case,
by applying \propref{ros} and Jensen's inequality, we have
\begin{equation}\label{a33a}
\begin{split}
{\mathbb{E}}|W|^{2p} & \le \kappa_{2p}\left(1+{\mathbb{E}}|T|^{2p-2}\right)\\
&\le 2\kappa_{2p}\le C_p\left(\dfrac{p}{\log p}\right)^{2p},
\end{split}
\end{equation}
and
\begin{equation}\label{a33b}
\begin{split}
{\mathbb{E}}|W|^{2p+2} & \le \kappa_{2p+2}\left(1+{\mathbb{E}}|T|^{2p}\right)\\
&\le 2\kappa_{2p+2}\le C_p\left(\dfrac{p}{\log p}\right)^{2p}.
\end{split}
\end{equation}
Since
\begin{equation*}
\begin{split}
    {\mathbb{P}}(W^* \le x) - \Phi(x) & = {\mathbb{E}}\{Wf_x(W) - W^*f_x(W^*)\} \\
    & = - {\mathbb{E}}\int_0^T g_x(W+t)\mathrm{d}t,
\end{split}
\end{equation*}
we have
\begin{equation}\label{star01}
\begin{split}
|{\mathbb{P}}(W^* \le x) - \Phi(x)| \le R_1+R_2,
\end{split}
\end{equation}
where
\begin{equation}\label{a34}
R_1= \left|{\mathbb{E}}\int_0^T g_x(W+t)\left({\textbf{1}}(W+t \le 0)+{\textbf{1}}\left(0<W+t \le  \dfrac{x}{2}\right)\right) \mathrm{d}t\right|
\end{equation}
and
\begin{equation}\label{a36}
R_2= \left|{\mathbb{E}}\int_0^T g_x(W+t){\textbf{1}}\left(W+t > \dfrac{x}{2}\right) \mathrm{d}t\right|.
\end{equation}
From the definition of $f_x$ and $g_x$, we have (see  \cite{ChenShao01})
\begin{equation}\label{CS3} g_x(w)=
\begin{cases}
\left(\sqrt{2\pi}(1+w^2)e^{w^2/2}(1-\Phi(w))-w\right)\Phi(x) &\text{ if  }w\ge x,\\
\left(\sqrt{2\pi}(1+w^2)e^{w^2/2}\Phi(w)+w\right)(1-\Phi(x)) &\text{ if  }w< x.
\end{cases}
\end{equation} 
\cite{ChenShao01} proved that $g_x\ge 0$, $g_x(w)\le
2(1-\Phi(x))$ for $w\le 0$, and $g_x$ is increasing for
$0\le w<x$.
From \eqref{CS3} and the fact that $\sqrt{2\pi}(1-\Phi(x))\le e^{-x^2/2}/x$ for all $x>0$, we have
\begin{equation}\label{star01b}
\begin{split}
g_x(x/2)&=\left(\sqrt{2\pi}\left(1+\dfrac{x^2}{4}\right)e^{x^2/8}\Phi(x/2)+\dfrac{x}{2}\right)(1-\Phi(x))\\
&\le \left(\dfrac{1}{x}+\dfrac{x}{4}\right)e^{-3x^2/8}+\dfrac{1}{2\sqrt{2\pi}}e^{-x^2/2}.
\end{split}
\end{equation}
For all $r\ge 1$, a straightforward calculation shows that
\[\max_{x>0}x^{r}e^{-x^2/2}<\max_{x>0}x^{r}e^{-3x^2/8}=\left(\dfrac{2\sqrt{r}}{\sqrt{3e}}\right)^r.\]
Therefore, from \eqref{a34} and \eqref{star01b}, we have
\begin{equation}\label{star01c}
\begin{split}
R_1 &\le {\mathbb{E}} \int_{0}^{|T|} \left(2(1-\Phi(x))+g_x(x/2)\right)\mathrm{d}t\\
&\le \dfrac{C_p}{1+x^p}\left(\dfrac{p}{\log p}\right)^p{\mathbb{E}}|T|\le \dfrac{C_p}{1+x^p}\left(\dfrac{p}{\log p}\right)^p \sqrt{{\mathbb{E}}T^2}.
\end{split}
\end{equation}
To bound $R_2$, we estimate
\begin{equation}\label{a37}
{\textbf{1}}\left(W+t > \dfrac{x}{2}\right)\le \dfrac{C_p}{1+x^p}\left(|W|^p+|T|^p\right) \text{ for all } 0\le t \le |T|.
\end{equation}
Combining \eqref{nu01} and \eqref{a37}, we have
\begin{equation}\label{a38}
\begin{split}
R_2&\le \dfrac{C_p}{1+x^p} {\mathbb{E}} \int_{0}^{|T|} (1+|W|+|T|)(|W|^p+|T|^p)\mathrm{d}t\\
& = \dfrac{C_p}{1+x^p}{\mathbb{E}}(1+|W|+|T|)(|W|^p|T|+|T|^{p+1}).
\end{split}
\end{equation}
We bound each term in \eqref{a38} as follows. Firstly, we have
\begin{equation}\label{a42}
{\mathbb{E}}|T|^{p+1}\le \sqrt{{\mathbb{E}}T^{2p+2}}\ \text{ and }\ {\mathbb{E}}|W||T|^{p+1}\le \sqrt{{\mathbb{E}}W^2 {\mathbb{E}}T^{2p+2}}= \sqrt{{\mathbb{E}}T^{2p+2}}.
\end{equation}
Secondly, by using the Cauchy–Schwarz inequality, \eqref{a33a} and \eqref{a33b}, and by noting that
${\mathbb{E}}|T|^{2p+2}< 1$, we have
\begin{equation}\label{a44}
\begin{split}
{\mathbb{E}}|W|^{p}|T|&\le \sqrt{{\mathbb{E}}|W|^{2p}{\mathbb{E}}T^2 }\le C_p\left(\dfrac{p}{\log p}\right)^p\sqrt{\mathbb{E}T^2},
\end{split}
\end{equation}
\begin{equation}\label{a45}
\begin{split}
{\mathbb{E}}|W|^{p+1}|T|&\le \sqrt{{\mathbb{E}}|W|^{2p+2}{\mathbb{E}}T^2 }\le  C_p\left(\dfrac{p}{\log p}\right)^p\sqrt{\mathbb{E}T^2},
\end{split}
\end{equation}
and
\begin{equation}\label{a43}
{\mathbb{E}}|T|^{p+2}\le \sqrt{{\mathbb{E}}T^2 {\mathbb{E}}T^{2p+2}}\le \sqrt{{\mathbb{E}}T^2}.
\end{equation}
Finally,
\begin{equation}\label{a46}
\begin{split}
{\mathbb{E}}|W|^{p}|T|^2&={\mathbb{E}}\left((|W|^{p+1}|T|)^{p/(p+1)}(|T|^{p+2})^{1/(p+1)}\right)\\
&\le {\mathbb{E}}\left(|W|^{p+1}|T|+|T|^{p+2}\right)\\
&\le  C_p\left(\dfrac{p}{\log p}\right)^p\sqrt{\mathbb{E}T^2},
\end{split}
\end{equation}
where we have used \eqref{a10} in the first inequality, and \eqref{a45} and \eqref{a43} in the second inequality.
From \eqref{a38}-\eqref{a46}, we have
\begin{equation}\label{a48}
\begin{split}
R_2&\le \dfrac{C_p}{1+x^p} \left(\dfrac{p}{\log p}\right)^p \left(\sqrt{\mathbb{E}T^2}+\sqrt{\mathbb{E}|T|^{2p+2}} \right).
\end{split}
\end{equation}
Combining \eqref{star01}, \eqref{star01c} and \eqref{a48}, we obtain \eqref{nou01}.
\end{proof}

\thmref{star} is a normal approximation for $W^{*}$. When $T=W^{*}-W$ has fast decaying tails, by using \thmref{star}, we can obtain useful bounds in normal approximation for $W$. This gives us the following theorem.

\begin{thm}\label{nonstar} Let $W$ be such that ${\mathbb{E}}W = 0$ and $\mathrm{Var}(W) = 1$, 
 and let $W^{*}$ be
$W$-zero biased and defined on the same probability space as
$W$. Let $T=W^{*}-W$ and $\varepsilon>0$ be arbitrary.

(i) We have
 \begin{equation}\label{bo01e}
 \sup_{x\in {\mathbb{R}}}|{\mathbb{P}}(W\le x)-\Phi(x)|\le \left(1+\dfrac{\sqrt{2\pi}}{4}\right)\sqrt{{\mathbb{E}}T^2}+\dfrac{\varepsilon}{\sqrt{2\pi}}+{\mathbb{P}}(|T|>\varepsilon).
\end{equation}

(ii) Let $p\ge 2$. Then for all $x \in \mathbb{R}$, 
\begin{equation}\label{nou01e} 
\begin{split}
&|{\mathbb{P}}(W \le x) - \Phi(x)|\\ 
&\le \dfrac{C_p}{1+|x|^p}\left(\dfrac{p}{\log p}\right)^p \left(\sqrt{{\mathbb{E}}T^2}+\sqrt{{\mathbb{E}}|T| ^{2p+2}}+\varepsilon+\sqrt{{\mathbb{P}}(|T|>\varepsilon)}\right).
\end{split}
 \end{equation}
 \end{thm}

\begin{rem}
{\rm
If $|T|\le \varepsilon$ almost surely, then \eqref{bo01e} reduces to
\begin{equation}\label{compare01}
\sup_{x\in {\mathbb{R}}}|{\mathbb{P}}(W\le x)-\Phi(x)|\le \left(1+\dfrac{\sqrt{2\pi}}{4}+\dfrac{1}{\sqrt{2\pi}}\right)\varepsilon.
\end{equation}
In Theorem 1.1 in \cite{Goldstein05}, the author considered the following distance between $W$ and the standard normal random variable $Z$
\[\operatorname{d}(W,Z)=\sup_{h\in{\mathcal{H}}}\left|\mathbb{E}h(W)-\mathbb{E}h(Z)\right|,\]
where $\mathcal{H}$ is a class of measurable functions on the real line which contains the collection of indicators of all half lines.
When $\mathcal{H}$ coincides with the collection of indicators of all half lines, the author proved that 
(see the first half of (10) in \cite{Goldstein05}) 
\begin{equation}\label{compare03}
\sup_{x\in {\mathbb{R}}}|{\mathbb{P}}(W\le x)-\Phi(x)|\le \left(127+12\varepsilon\right)\varepsilon.
\end{equation}

}
\end{rem} 

\begin{proof}[Proof of \thmref{nonstar}]
Let $\varepsilon>0$ be arbitrary. Then by \eqref{bo01}, we have
\begin{equation}\label{bo06}
\begin{split}
{\mathbb{P}}(W\le x)-\Phi(x)&={\mathbb{P}}(W^{*}\le x+W^{*}-W)-\Phi(x)\\
& \le {\mathbb{P}}(W^{*}\le x+\varepsilon)-\Phi(x+\varepsilon)+\Phi(x+\varepsilon)-\Phi(x)\\
&\qquad+{\mathbb{P}}(W^{*}-W>\varepsilon)\\
& \le \left(1+\dfrac{\sqrt{2\pi}}{4}\right)\sqrt{{\mathbb{E}}T^2}+\dfrac{\varepsilon}{\sqrt{2\pi}}+{\mathbb{P}}(|W^{*}-W|>\varepsilon),
\end{split}
\end{equation}
and 
\begin{equation}\label{bo07}
\begin{split}
{\mathbb{P}}(W\le x)-\Phi(x)& \ge {\mathbb{P}}(W^{*}\le x-\varepsilon)-\Phi(x-\varepsilon)+\Phi(x-\varepsilon)-\Phi(x)\\
&\qquad -{\mathbb{P}}(W^{*}-W<-\varepsilon)\\
& \ge -\left(1+\dfrac{\sqrt{2\pi}}{4}\right)\sqrt{{\mathbb{E}}T^2}-\dfrac{\varepsilon}{\sqrt{2\pi}}-{\mathbb{P}}(|W^{*}-W|>\varepsilon).
\end{split}
\end{equation}
Combining \eqref{bo06} and \eqref{bo07}, we obtain \eqref{bo01e}.

To prove \eqref{nou01e}, it suffices to consider
$x\ge2$, as in the proof of \eqref{nou01}.  Similar to the proof of \eqref{a32}, we have
\begin{equation}\label{a50}
\begin{split}
&|P(W\le x)-\Phi(x)| \le \dfrac{C_p}{1+x^p}\left(\dfrac{p}{\log p}\right)^p \left(1+\sqrt{{\mathbb{E}}|T|^{2p+2}}\right).
\end{split}
\end{equation}
Therefore, if either ${\mathbb{E}}|T|^{2p+2}\ge 1$ or $\varepsilon\ge 1$, then \eqref{nou01e} holds. It 
remains to consider the case where ${\mathbb{E}}|T|^{2p+2}< 1$ and $\varepsilon<1$. In this case, similar to \eqref{a33a}, we have
\begin{equation}
\label{nou02}
{\mathbb{E}}|W^{*}|^{2p}=\dfrac{{\mathbb{E}}|W|^{2p+2}}{(2p+1){\mathbb{E}}W^2}\le \dfrac{2\kappa_{2p+2}}{2p+1}\le C_p \left(\dfrac{p}{\log p}\right)^{2p}.
\end{equation}
Since 
\begin{equation*}
\begin{split}
{\mathbb{P}}(W^*>x+\varepsilon)&={\mathbb{P}}(W^*>x+\varepsilon,T>\varepsilon)+{\mathbb{P}}(W^*>x+\varepsilon,T\le \varepsilon)\\
&\le {\mathbb{P}}(W^*>x,T>\varepsilon)+{\mathbb{P}}(W>x),
\end{split}
\end{equation*}
we have
\begin{equation}\label{nu02}
\begin{split}
 {\mathbb{P}}(W\le x)-\Phi(x)&=1-{\mathbb{P}}(W>x)-\Phi(x)\\
 &\le 1-{\mathbb{P}}(W^{*}> x+\varepsilon)-\Phi(x)+{\mathbb{P}}(W^{*}>x,T>\varepsilon).
\end{split}
\end{equation}
Combining \eqref{nou01}, \eqref{nou02} and \eqref{nu02}, we have
\begin{equation}\label{nou03}
\begin{split}
 &{\mathbb{P}}(W\le x)-\Phi(x) \le {\mathbb{P}}(W^{*}\le x+\varepsilon)-\Phi(x+\varepsilon)\\
 &\quad+\Phi(x+\varepsilon)-\Phi(x)+{\mathbb{P}}(W^{*}>x,T>\varepsilon)\\
& \le \dfrac{C_{p}}{1+x^p}\left(\dfrac{p}{\log p}\right)^p \left(\sqrt{{\mathbb{E}}T^2}+\sqrt{{\mathbb{E}}|T| ^{2p+2}}\right)\\
&\quad+\dfrac{\varepsilon e^{-x^2/2}}{\sqrt{2\pi}}+\sqrt{{\mathbb{P}}(|T|>\varepsilon)}\sqrt{{\mathbb{P}}(|W^{*}|>x)}\\
& \le \dfrac{C_{p}}{1+x^p}\left(\dfrac{p}{\log p}\right)^p \left(\sqrt{{\mathbb{E}}T^2}+\sqrt{{\mathbb{E}}|T| ^{2p+2}}+\varepsilon\right)\\
&\quad+\dfrac{\sqrt{{\mathbb{P}}(|T|>\varepsilon)}\sqrt{{\mathbb{E}}|W^{*}|^{2p}}}{x^p}\\
& \le \dfrac{C_{p}}{1+x^p}\left(\dfrac{p}{\log p}\right)^p \left(\sqrt{{\mathbb{E}}T^2}+\sqrt{{\mathbb{E}}|T| ^{2p+2}}+\varepsilon+\sqrt{{\mathbb{P}}(|T|>\varepsilon)}\right).
\end{split}
\end{equation}
Similarly, by noting that $x-\varepsilon>x-1\ge 1$, we can show that
\begin{equation}\label{nou04}
\begin{split}
&{\mathbb{P}}(W\le x)-\Phi(x)\\
& \ge -\dfrac{C_{p}}{1+x^p}\left(\dfrac{p}{\log p}\right)^p \left(\sqrt{{\mathbb{E}}T^2}+\sqrt{{\mathbb{E}}|T| ^{2p+2}}+\varepsilon+\sqrt{{\mathbb{P}}(|T|>\varepsilon)}\right).
\end{split}
\end{equation}
Combining \eqref{nou03} and \eqref{nou04}, we obtain \eqref{nou01e}.
\end{proof}

\section{Proofs of the main results}\label{sec:Proof}
The rate in the following proposition is better than that of \thmref{Ja1} in the case where
$\alpha\ge n^{1+\delta}$ for some $\delta>0$ fixed. We would like to note here that
when $1\le \alpha \le n/\log^2 n$  or  $\alpha\ge n^{1+\delta}$ for some $\delta>0$ fixed,
the convergence rate obtained in
\propref{Jack} is exactly the rate in Fulman's conjecture. \cite{CGR} also obtained the bound $O(\sqrt{\alpha}/n)$ for the case
$\alpha\ge n^{1+\delta}$ by applying induction with Stein's method. 

\begin{prop}\label{Jack} Let $n\ge 3$ be an integer. Let $\alpha\ge1$ and $W_{n,\alpha}$ be as in \eqref{J00}. 
	Then
	\begin{equation}\label{J11}
	\sup_{x\in {\mathbb{R}}}|{\mathbb{P}}_{\alpha}(W_{n,\alpha}\le
		x)-\Phi(x)|\le 8.2\max\left\{\dfrac{1}{\sqrt{n}},\dfrac{\sqrt{\alpha}\log n}{n}\right\}.
\end{equation} 
If, in addition, 
$\alpha\ge n^{1+\delta}$ for some $\delta:=\delta(\alpha,n)>0$, then
		\begin{equation}\label{J13}\sup_{x\in {\mathbb{R}}}|{\mathbb{P}}_{\alpha}(W_{n,\alpha}\le
		x)-\Phi(x)|\le \left(4.7+\dfrac{3.1}{\delta}\right)\dfrac{\sqrt{\alpha}}{n}.
		\end{equation}
\end{prop}

\begin{rem}{\rm
If $\alpha>n$, then we can write $\alpha=n^{1+\delta}$, where 
\[\delta=\dfrac{\log\alpha-\log n}{\log n}>0.\]
Applying \eqref{J13}, we have
		\begin{equation}\label{J14}
\sup_{x\in {\mathbb{R}}}|{\mathbb{P}}_{\alpha}(W_{n,\alpha}\le x)-\Phi(x)|
\le \dfrac{4.7\log\alpha-1.6\log n}{\log\alpha-\log n}\dfrac{\sqrt{\alpha}}{n}.
		\end{equation}
We make some notes as follows.
\begin{description}
\item (i) If $\alpha \sim Kn$ for some $K>1$ fixed, then the rate
obtained in \eqref{J14} is $O\left(\dfrac{\sqrt{\alpha} \log n}{n}\right)$
which is the same as the rate obtained in \eqref{J11}.

\item (ii) If $\alpha \sim n\left(\log n\right)^K $ for some $K>0$ fixed, then the rate obtained in \eqref{J14} is 
$O\left(\dfrac{\sqrt{\alpha}\log n }{n\log(\log n)}\right)$
which is better than the rate obtained in \eqref{J11}.

\item (iii) If $\alpha\ge n^{1+\delta}$ for some $\delta>0$ fixed,
then the convergence rate obtained in
\eqref{J14} is $O\left(\dfrac{\sqrt{\alpha}}{n}\right)$ which is exactly the rate in Fulman's conjecture.

\end{description}

}
\end{rem}

We will prove \propref{Jack} by applying \thmref{nonstar}. In
\cite{Kerov00}, the author proved that there is a growth process giving a sequence of partitions
$(\lambda(1),\dots,\lambda(n))$ with $\lambda(j)$ distributed
according to the Jack$_\alpha$ measure on partitions of size $j$. We
refer to \cite{Fulman04} for details. Given Kerov's process, let
$X_{1,\alpha}=0$, $X_{j,\alpha}=c_{\alpha}(a)$ where $a$ is the box
added to $\lambda(j-1)$ to obtain $\lambda(j)$ and the
``$\alpha$-content'' $c_\alpha(a)$ of a box $a$ is defined to be
$\alpha(\text{column number of }a-1)- (\text{row number of }a-1)$, $j\ge 2$. Then one can write (see \cite{Fulman06b, FulmanGoldstein})
\begin{equation}\label{Fu01}W_{n,\alpha}=\dfrac{\sum_{j=1}^n X_{j,\alpha}}{\sqrt{\alpha{n
\choose 2}}}.\end{equation} Therefore, constructing $\nu$ from the
Jack$_\alpha$ measure on partitions of $n-1$ and then taking
one step in Kerov's growth process yields $\lambda$ with the
Jack$_\alpha$ measure on partitions of $n$, we have \begin{equation}\label{FG6}
W_{n,\alpha}=V_{n,\alpha}+\eta_{n,\alpha},\end{equation} where
\begin{equation}\label{FG7}
V_{n,\alpha}=\dfrac{\sum_{x\in\nu}c_{\alpha}(x)}{\sqrt{\alpha{n\choose
2}}}= \sqrt{\dfrac{n-2}{n}} W_{n-1,\alpha},\text{  }
\eta_{n,\alpha}=\dfrac{X_{n,\alpha}}{\sqrt{\alpha{n \choose
2}}}=\dfrac{c_{\alpha}(\lambda/\nu)}{\sqrt{\alpha{n \choose 2}}},
\end{equation} 
and $c_{\alpha}(\lambda/\nu)$ denotes the $\alpha$-content of the box
added to $\nu$ to obtain $\lambda$.
\cite{Fulman06b} proved that
\begin{equation}\label{F11z} EW_{n,\alpha}=0,\ EW_{n,\alpha}^2=1,\end{equation} 
\begin{equation}\label{F11} E\eta_{n,\alpha}=0,\ E\eta_{n,\alpha}^2=
\dfrac{2}{ n},
\end{equation}  
and
\begin{equation}\label{F13}E\eta_{n,\alpha}^4=\dfrac{2}{
n^2}\left(\dfrac{4n-6}{n-1}+\dfrac{2(\alpha-1)^2}{\alpha(n-1)}\right).
\end{equation}  
From Theorems 3.1 and 4.1 in
\cite{FulmanGoldstein}, there
exists a random variable $\eta^{*}_{n,\alpha}$ defined on the same probability
space with $\eta_{n,\alpha}$, and satisfying that
$\eta^{*}_{n,\alpha}$ has $\eta_{n,\alpha}$-zero biased distribution
and that \begin{equation}\label{FG9} W^{*}_{n,\alpha}=V_{n,\alpha}+\eta^{*}_{n,\alpha}\end{equation}
has $W_{n,\alpha}$-zero biased distribution. Here and thereafter, we denote 
\[T_{n,\alpha}=\eta_{n,\alpha}-\eta_{n,\alpha}^{*}.\]

The following lemma gives a bound for ${\mathbb{E}}(\eta_{n,\alpha}^{*})^2$.

\begin{lem}\label{Ful12}
For $\alpha\ge 1$, we have 
\begin{equation}\label{F15}
\begin{split} 
{\mathbb{E}}(\eta_{n,\alpha}^{*})^2&=\dfrac{1}{
3n}\left(\dfrac{4n-6}{n-1}+\dfrac{2(\alpha-1)^2}{\alpha(n-1)}\right)\\
&\le \dfrac{1}{
3n}\left(4+\dfrac{2\alpha}{n-1}\right).
\end{split}
\end{equation}
\end{lem}
\begin{proof}
Applying \eqref{zer} with $f(x)=x^3$,
 we have
\begin{equation}\label{F15b} {\mathbb{E}}(\eta_{n,\alpha}^{*})^2=\dfrac{{\mathbb{E}}(\eta_{n,\alpha})^4}{3 {\mathbb{E}}\eta_{n,\alpha}^{2}}.
\end{equation}
Combining \eqref{F11}, \eqref{F13} and \eqref{F15b}, we obtain  
\eqref{F15}.
 \end{proof}

For a partition
$\lambda$ of a positive integer $n$, we recall that the length of row $i$ of
$\lambda$ and the length of column $i$ of $\lambda$ are denoted by
$\lambda_i$ and $\lambda_{i}^{'}$, respectively.

From a computation in the proof of Lemma 6.6 in \cite{Fulman04} and Stirling's formula, we have the following lemma.

\begin{lem}\label{Fu92}
Let $\alpha>0$. Then for $1\le l\le n$, we have
\begin{equation}\label{Fu90x}
{\mathbb{P}}_{\alpha}(\lambda_1=l)\le\dfrac{\alpha}{2\pi}\left(\dfrac{ne^2}{\alpha l^2}\right)^l.
\end{equation}
\end{lem}
\begin{proof}
It is proved by \cite{Fulman04} that
\begin{equation}\label{Fu70}{\mathbb{P}}_{\alpha}(\lambda_1=l)\le
\left(\dfrac{n}{\alpha}\right)^l\dfrac{\alpha l}{(l!)^2}.\end{equation}
By Stirling's formula, we have for all $l\ge 1$,
\begin{equation}\label{Fu70a}
l!\ge\sqrt{2\pi l}\left(\dfrac{l}{ e}\right)^{l}.
\end{equation}
Combining \eqref{Fu70} and \eqref{Fu70a}, we have \eqref{Fu90x}.
\end{proof}

In order to apply \thmref{nonstar}, we need to bound ${\mathbb{P}}(|T_{n,\alpha}|>\varepsilon)$ for suitably chosen $\varepsilon$.
The following lemma shows that $\left|T_{n,\alpha}\right|$ has a very light tail.

\begin{lem}\label{Ful13}
For all $\alpha\ge 1$ and $q> 1$, we have 
\begin{equation*}
\begin{split} 
{\mathbb{P}}_{\alpha}\left(|T_{n,\alpha}|> \dfrac{2e\sqrt{2q}}{\sqrt{n-1}}\right)
&\le \dfrac{
	\alpha}{\pi(q-1)q^{e \sqrt{qn/\alpha}}}+\dfrac{\alpha^2 q\left( e \sqrt{qn/\alpha}(q-1)+q+1\right)}{\pi(n-1)(q-1)^3q^{ e \sqrt{qn/\alpha}}}.
\end{split}
\end{equation*}
\end{lem}
\begin{proof}
First, we take an arbitrary $\alpha>0$.
It follows from  \eqref{Fu90x} that
\begin{equation}\label{Fu84zz}
{\mathbb{P}}_{\alpha}(\lambda_1=k+1)\le \dfrac{\alpha}{2\pi q^{k+1}}
\end{equation}
for all $k\ge
e\sqrt{qn/\alpha}$. Therefore
\begin{equation}\label{Fu71}
\begin{split}
{\mathbb{P}}_{\alpha}\left(\lambda_1-1>  e \sqrt{qn/\alpha}\right)
&={\mathbb{P}}_{\alpha}\left(\lambda_1-1 \ge \lfloor e \sqrt{qn/\alpha}\rfloor +1\right)\\
&=\sum_{k\ge \lfloor e \sqrt{qn/\alpha}\rfloor+1} {\mathbb{P}}_{\alpha}(\lambda_1=k+1)\\
&\le \dfrac{\alpha}{2\pi}\sum_{k\ge \lfloor  e \sqrt{qn/\alpha}\rfloor +1}\dfrac{1}{q^{k+1}}\\
&=\dfrac{q\alpha}{2\pi (q-1)q^{\lfloor  e \sqrt{qn/\alpha}\rfloor +2}}\\
& \le \dfrac{\alpha}{2\pi (q-1) q^{e \sqrt{qn/\alpha}}}.
\end{split}
\end{equation}
We note that from the definition of Jack measure, ${\mathbb{P}}_{\alpha}(\lambda)={\mathbb{P}}_{1/\alpha}(\lambda^t)$, where $\lambda^t$ is the transpose
 partition of $\lambda$. 
Applying
\eqref{Fu71} with $\alpha$ replaced by $1/\alpha$, we have
\begin{equation}\label{Fu72}
{\mathbb{P}}_{\alpha}(\lambda_{1}^{'}-1> e\sqrt{q\alpha n})\le
\dfrac{1}{{2\pi}\alpha (q-1)q^{e\sqrt{q\alpha n}}}.
\end{equation}
Since $|X_{n,\alpha}|\le
\max\{\alpha(\lambda_1-1),\lambda_{1}^{'}-1\}$, it follows from \eqref{Fu71} and \eqref{Fu72} that
\begin{equation}
\begin{split}
{\mathbb{P}}_{\alpha}\left(|\eta_{n,\alpha}|>
\dfrac{e\sqrt{2q}}{\sqrt{n-1}}\right)
&={\mathbb{P}}_{\alpha}\left(\dfrac{\sqrt{2}|X_{n,\alpha}|}{\sqrt{\alpha n(n-1)}}>
\dfrac{e\sqrt{2q}}{\sqrt{n-1}}\right)\\
&\le {\mathbb{P}}_{\alpha}\left(\max\{\alpha(\lambda_1-1),\lambda_{1}^{'}-1\}>
e\sqrt{q\alpha n}\right)\\
&\le {\mathbb{P}}_{\alpha}\left(\lambda_1-1> e\sqrt{qn/\alpha }\right)
+{\mathbb{P}}_{\alpha}\left(\lambda_{1}^{'}-1> e\sqrt{q\alpha n}\right)\\
&\le \dfrac{
\alpha}{{2\pi}(q-1)q^{e \sqrt{qn/\alpha}}}+\dfrac{
1}{2\pi\alpha (q-1)q^{e\sqrt{q\alpha n}}}.
\end{split}
\end{equation}
For $\alpha\ge 1$, it reduces to
\begin{equation}\label{a70}
\begin{split}
{\mathbb{P}}_{\alpha}\left(|\eta_{n,\alpha}|>
\dfrac{e\sqrt{2q}}{\sqrt{n-1}}\right)
&\le \dfrac{
	\alpha}{{\pi}(q-1)q^{e \sqrt{qn/\alpha}}}.
\end{split}
\end{equation}

Recall that if $X$ is a random variable with ${\mathbb{E}}X=0$, ${\mathbb{E}}X^2=\sigma^2$ 
and if $X^*$ has $X$-zero-biased distribution, then for $x>0$, applying \eqref{zer} with $f_x(w)=(w-x)\textbf{1}(w> x)$, we have
\begin{equation}\label{Fu001}
{\mathbb{P}}(X^*> x)={\mathbb{E}}[X(X-x){\textbf{1}}(X> x)]/\sigma^2.
\end{equation}
By using \eqref{Fu001} and \eqref{Fu84zz}, and 
noting that
\[\eta_{n,\alpha}\le \dfrac{\sqrt{2\alpha}(\lambda_1-1)}{\sqrt{n(n-1)}},\]
 we have 
\begin{equation}\label{F97x}
\begin{split}
&{\mathbb{P}}_{\alpha}\left(\eta_{n,\alpha}^{*}> \dfrac{e\sqrt{2q}}{\sqrt{n-1}}\right)\\
&\quad= \dfrac{n}{2} {\mathbb{E}}\left(\eta_{n,\alpha}\left(\eta_{n,\alpha}-\dfrac{e\sqrt{2q}}{\sqrt{n-1}} \right) {\textbf{1}}
\left(\eta_{n,\alpha}> \dfrac{e\sqrt{2q}}{\sqrt{n-1}}\right)\right)\\
&\quad\le \dfrac{\alpha}{n-1} {\mathbb{E}}\left((\lambda_1-1)\left(\lambda_1-1-e \sqrt{qn/\alpha}\right) {\textbf{1}}
\left(\lambda_1-1>  e \sqrt{qn/\alpha}\right)\right)\\
&\quad \le \dfrac{\alpha}{n-1} \sum_{k=1}^\infty k(k+\lfloor  e \sqrt{qn/\alpha}\rfloor){\mathbb{P}}_{\alpha}\left(\lambda_1=k+\lfloor  e \sqrt{qn/\alpha}\rfloor+1\right)\\
&\quad \le \dfrac{\alpha^2}{2\pi(n-1)} \sum_{k=1}^\infty \dfrac{k\left(k+\lfloor  e \sqrt{qn/\alpha}\rfloor\right)}{q^{k+\lfloor  e \sqrt{qn/\alpha}\rfloor+1}}\\
&\quad= \dfrac{\alpha^2\left(\lfloor  e \sqrt{qn/\alpha}\rfloor(q-1)+q+1\right)}{2\pi(n-1)(q-1)^3q^{\lfloor  e \sqrt{qn/\alpha}\rfloor}} \\
&\quad \le \dfrac{\alpha^2 q\left( e \sqrt{qn/\alpha}(q-1)+q+1\right)}{2\pi(n-1)(q-1)^3q^{ e \sqrt{qn/\alpha}}}.
\end{split}
\end{equation} 
Applying  \eqref{Fu90x} again, we have
\begin{equation}\label{Fu90xx}
{\mathbb{P}}_{\alpha}(\lambda_{1}^{'}=k+1)=
{\mathbb{P}}_{1/\alpha}(\lambda_{1}=k+1)\le \dfrac{1}{2\pi \alpha q^{k+1}}
\end{equation}
for all $k\ge
e\sqrt{q\alpha n}$. 
By using \eqref{Fu90xx} and noting that
\[\eta_{n,\alpha}\ge - \dfrac{\sqrt{2}(\lambda_{1}^{'}-1)}{\sqrt{\alpha n(n-1)}},\]
we have 
\begin{equation}\label{F97y}
\begin{split} 
&{\mathbb{P}}_{\alpha}\left(-\eta_{n,\alpha }^{*}> \dfrac{e\sqrt{2q}}{\sqrt{n-1}}\right)\\
& = \dfrac{n}{2} {\mathbb{E}}\left(-\eta_{n,\alpha}\left(-\eta_{n,\alpha}-\dfrac{e\sqrt{2q}}{n-1} \right){\textbf{1}}
\left(-\eta_{n,\alpha}>\dfrac{e\sqrt{2q}}{\sqrt{n-1}}\right)\right]\\
& \le \dfrac{1}{\alpha(n-1)} {\mathbb{E}}\left((\lambda_{1}^{'}-1)(\lambda_{1}^{'}-1-e\sqrt{q\alpha n}){\textbf{1}}
\left(\lambda_{1}^{'}-1>e\sqrt{q\alpha n }\right)\right)\\
& \le \dfrac{1}{\alpha(n-1)}\sum_{k=1}^\infty k(k+\lfloor e\sqrt{q\alpha n }\rfloor){\mathbb{P}}_{\alpha}\left(\lambda_{1}^{'}=k+\lfloor e\sqrt{q\alpha n }\rfloor+1\right)\\
& \le \dfrac{1}{2\pi\alpha^2(n-1)}\sum_{k=1}^\infty \dfrac{k\left(k+\lfloor e\sqrt{q\alpha n }\rfloor\right)}{q^{k+\lfloor e\sqrt{q\alpha n }\rfloor+1}}\\
&=\dfrac{\lfloor e\sqrt{q\alpha n }\rfloor(q-1)+q+1}{2\pi\alpha^2(n-1)(q-1)^3 q^{\lfloor e\sqrt{q\alpha n }\rfloor}}\\
& \le \dfrac{q\left(e\sqrt{q\alpha n }(q-1)+q+1\right)}{2\pi\alpha^2(n-1)(q-1)^3 q^{ e\sqrt{q\alpha n }}}.
\end{split}
\end{equation}
For $\alpha\ge 1$, \eqref{F97x} and \eqref{F97y} reduce to
\begin{equation}\label{F97z}
\begin{split}
{\mathbb{P}}_{\alpha}\left(\left|\eta_{n,\alpha}^{*}\right|>\dfrac{e\sqrt{2q}}{\sqrt{n-1}}\right)
&\le \dfrac{\alpha^2 q\left( e \sqrt{qn/\alpha}(q-1)+q+1\right)}{\pi(n-1)(q-1)^3q^{ e \sqrt{qn/\alpha}}}.
\end{split}
\end{equation} 
The conclusion of the lemma follows from \eqref{a70} and \eqref{F97z}. The proof of the lemma is completed.
\end{proof}

\begin{proof}[Proof of \propref{Jack}]
It suffices to consider $x\ge0$ since we can simply apply the result to $-W_{n,\alpha}$ when $x<0$.
For a random variable $W$ with ${\mathbb{E}}W=0$ and ${\mathrm{Var}}(W)=1$, 
\cite{ChenShao01} proved that
\begin{equation}\label{c02}
\sup_{x\ge0}|{\mathbb{P}}(W\le x)-\Phi(x)|\le
\sup_{x\ge0}\Big|\dfrac{1}{1+x^2}-(1-\Phi(x))\Big|\le 0.55.
\end{equation}

Firstly, we prove \eqref{J11}. From \eqref{c02}, it suffices to prove the proposition for $n\ge 200$. 
Let 
\begin{equation}\label{c01}
K\ge e^{1/4},\ q=K^2\max\left\{1,\dfrac{\alpha\log^2n}{n}\right\},\ \varepsilon=\dfrac{2e\sqrt{2q}}{\sqrt{n-1}}.
\end{equation}
Since $q\ge e^{1/2}$, it is clear that for $x>2$, $x/q^x$ is decreasing in $x$. Therefore, 
by applying \lemref{Ful13} with noting that $qn/\alpha\ge K^2\log^2n$, we have
\begin{equation}\label{b01}
\begin{split}
{\mathbb{P}}_{\alpha}\left(|T_{n,\alpha}|> \dfrac{2e\sqrt{2q}}{\sqrt{n-1}}\right)&\le
\dfrac{qn}{\pi (q-1)q^{Ke\log n} K^2\log^2 n}\\
&\quad+\dfrac{n^2 q^3\left(e (q-1)K\log n+q+1\right)}{\pi (n-1)(q-1)^3q^{e K\log n}K^2\log^4n}:=f(q).
\end{split}
\end{equation}
Since $e K\log n>5$, $f(q)$ is decreasing on $(1,\infty)$. Therefore
\begin{equation}\label{b03}
\begin{split}
f(q)\le f(K^2)&=\dfrac{K^2n}{\pi (K^2-1)K^{2eK\log n}K^2\log^2 n}\\
&+\dfrac{n^2 K^6 \left(e(K^2-1)K\log n+K^2+1\right)}{\pi(n-1)(K^2-1)^3K^{2eK\log n}K^4\log^4n}.
\end{split}
\end{equation}
By choosing $K=e^{1/4}$ and noting that $n>200$, we have
\begin{equation}\label{c03}
\begin{split}
{\mathbb{P}}_{\alpha}\left(|T_{n,\alpha}|>\varepsilon\right)
&\le f(e^{1/2})\\
&=\dfrac{\sqrt{e} n}{\pi\sqrt{e} (\sqrt{e}-1)n^{0.5e^{5/4}}\log^2 n}\\
&\quad+\dfrac{n^2 e^{3/2} \left(e^{5/4}(\sqrt{e}-1)\log n+\sqrt{e}+1\right)}{\pi(n-1)(\sqrt{e}-1)^3n^{0.5e^{5/4}}\log^4n}\\
&\le \dfrac{0.05}{\sqrt{n}},
\end{split}
\end{equation}
and
\begin{equation}\label{c05}
\begin{split}
\dfrac{\varepsilon }{\sqrt{2\pi}}
&= \dfrac{2e^{5/4}\sqrt{n}}{\sqrt{\pi(n-1)}}\max\left\{\dfrac{1}{\sqrt{n}},\dfrac{\sqrt{\alpha}\log n}{n}\right\}\\
&\le 3.95 \max\left\{\dfrac{1}{\sqrt{n}},\dfrac{\sqrt{\alpha}\log n}{n}\right\}.
\end{split}
\end{equation}
By \eqref{F11} and \eqref{F15}, we have
\begin{equation}\label{F19}
\begin{split}
\sqrt{{\mathbb{E}}T_{n,\alpha}^2}&=\sqrt{{\mathbb{E}}(\eta_{n,\alpha}^{*}-\eta_{n,\alpha})^2}\\
&\le \sqrt{{\mathbb{E}}(\eta_{n,\alpha}^{*})^2}+\sqrt{{\mathbb{E}}(\eta_{n,\alpha})^2}\\
&\le \left(\dfrac{4}{3n}+\dfrac{2 \alpha}{3n(n-1)} \right)^{1/2}+\left(\dfrac{2}{n}\right)^{1/2}\\
&\le \left(\left(\dfrac{4}{3}+\dfrac{2n}{3(n-1) \log^2 n} \right)^{1/2} +\sqrt{2} \right)\max\left\{\dfrac{1}{\sqrt{n}},\dfrac{\sqrt{\alpha}\log n}{n}\right\}.
\end{split}
\end{equation}
Since $n> 200$, it follows from \eqref{F19} that
\begin{equation}\label{b05}
\begin{split}
\left(1+\dfrac{\sqrt{2\pi}}{4}\right)\sqrt{ {\mathbb{E}}T_{n,\alpha}^2}&\le 
4.2\max\left\{\dfrac{1}{\sqrt{n}},\dfrac{\sqrt{\alpha}\log n}{n}\right\}.
\end{split}
\end{equation}
Apply \thmref{nonstar} (i), \eqref{J11} follows from \eqref{c03}, \eqref{c05} and \eqref{b05}.

Now we prove \eqref{J13}. If either $\delta\ge 1$ or $0<\delta<1$ and $n\le 200$,
then \eqref{J13} holds by \eqref{c02}. Therefore we may assume that $0<\delta<1$ and $n> 200$. Let 
\begin{equation}\label{b12}
0<L\le 1,\ q=\dfrac{\alpha}{(L\delta)^2 n},\ \varepsilon'=\dfrac{2e\sqrt{2q}}{\sqrt{n-1}}.
\end{equation}
Since $n>200$ and $0<\delta<1$, elementary calculus shows that 
\[q\ge \dfrac{n^\delta}{\delta^2}>51.\]
By applying \lemref{Ful13}, we have
\begin{equation}\label{b13}
\begin{split}
&{\mathbb{P}}_{\alpha}\left(|T_{n,\alpha}|>\varepsilon' \right)\le
\dfrac{\alpha}{\pi(q-1)q^{e/(L\delta)}}+
\dfrac{(L\delta)^2 q^2 n \alpha\left(e(q-1)/(L\delta)+q+1\right)}{\pi (n-1)(q-1)^3q^{e/(L\delta)}}.
\end{split}
\end{equation}
By choosing $L=1$ and noting $n>200,q>51$, we have from \eqref{b13} that
\begin{equation}\label{b14}
\begin{split}
{\mathbb{P}}_{\alpha}\left(|T_{n,\alpha}|>\varepsilon' \right)
&\le \dfrac{0.08\sqrt{\alpha }}{n},
\end{split}
\end{equation}
and
\begin{equation}\label{b16}
\begin{split}
\dfrac{\varepsilon'}{\sqrt{2\pi}}&= \dfrac{2\sqrt{n} e  }{\sqrt{\pi(n-1)} \delta}\dfrac{\sqrt{\alpha}}{n}
\le \dfrac{3.1}{\delta}\dfrac{\sqrt{\alpha}}{n}.
\end{split}
\end{equation}
Using the second inequality in \eqref{F19} and noting again that $\alpha>n>200$, we also have
\begin{equation}\label{b15}
\begin{split}
\sqrt{{\mathbb{E}}T_{n,\alpha}^2}
&\le \left(\dfrac{4}{3n}+\dfrac{2 \alpha}{3n(n-1)} \right)^{1/2}+\left(\dfrac{2}{n}\right)^{1/2}\\
&\le \left(\sqrt{\dfrac{4}{3}+\dfrac{2\times 201}{3\times 200}}+\sqrt{2}\right) \dfrac{\sqrt{\alpha}}{n}.
\end{split}
\end{equation}
It follows from \eqref{b15} that
\begin{equation}\label{b18}
\begin{split}
\left(1+\dfrac{\sqrt{2\pi}}{4}\right)\sqrt{ {\mathbb{E}}T_{n,\alpha}^2}&\le 
4.62\dfrac{\sqrt{\alpha}}{n}.
\end{split}
\end{equation}
Apply \thmref{nonstar} (i) with $\varepsilon'$ plays the role of $\varepsilon$ in \thmref{nonstar} (i), \eqref{J13} follows from \eqref{b14}, \eqref{b16} and \eqref{b18}.
\end{proof}

The following proposition establishes non-uniform bounds on the Kolmogorov distance for Jack measures.
\begin{prop}\label{nonuni-Jack}
Let $n\ge 3$ be an integer. Let $p\ge 2$, $1\le \alpha<n^{2}$ and $W_{n,\alpha}$ be as in \eqref{J00}.
Then for all  $x\in {\mathbb{R}}$, we have
	\begin{equation}\label{n02}
	\left|{\mathbb{P}}_{\alpha}(W_{n,\alpha}\le
	x)-\Phi(x)\right|\le \dfrac{C_p}{1+|x|^p}\left(\dfrac{p^2}{\log p} \right)^p\max\left\{\dfrac{1}{\sqrt{n}},\dfrac{\sqrt{\alpha}\log n}{n} \right\}.
	\end{equation}
	
If, in addition, there exist $\delta:=\delta(\alpha,n)>0$ such that $n^{1+\delta}\le \alpha< n^2$, then for all  $x\in {\mathbb{R}}$, we have
\begin{equation}\label{n03}|{\mathbb{P}}_{\alpha}(W_{n,\alpha}\le
x)-\Phi(x)|\le \left(1+\dfrac{1}{\delta^{p+1}}\right)\dfrac{C_p}{1+|x|^p}\left(\dfrac{p^2}{\log p} \right)^p \dfrac{\sqrt{\alpha}}{n}.
\end{equation}
\end{prop}

\begin{proof}

We observe that if $\alpha>n^{1+\delta},\ 1/4\le\delta<1$, then \eqref{n03} implies \eqref{n02} 
(the value $1/4$ is choosen for convenience only).
Therefore, we only need to prove \eqref{n02} for the case where $1\le \alpha<n^{5/4}$.
For $n\ge 3$ and
$1\le \alpha< n^{5/4}$, we have
\begin{equation}\label{c21}
\begin{split}
\max\left\{\dfrac{1}{\sqrt{n}},\dfrac{\sqrt{\alpha}\log n}{n}\right\}<1.
\end{split}
\end{equation}
Let $K=p+2$ and let $q,\varepsilon$ be as in \eqref{c01}. 
Then
\begin{equation}\label{c22}
\begin{split}
\varepsilon & \le 20(p+2) \max\left\{\dfrac{1}{\sqrt{n}},\dfrac{\sqrt{\alpha}\log n}{n}\right\}.
\end{split}
\end{equation}
From \eqref{b03}, we have
\begin{equation}\label{c23}
\begin{split}
{\mathbb{P}}_{\alpha}\left(|T_{n,\alpha}|>\varepsilon\right)
&\le \dfrac{(p+2)^2n}{\pi ((p+2)^2-1)n^{2e(p+2)}(p+2)^2\log^2 n}\\
&\quad+\dfrac{n^2 (p+2)^6 \left(e(p+2)((p+2)^2-1)\log n+(p+2)^2+1\right)}{\pi(n-1)((p+2)^2-1)^3n^{2e(p+2)}(p+2)^4\log^4n}\\
&\le \dfrac{C_p}{n^{2e(p+2)-1}}.
\end{split}
\end{equation}
To apply \thmref{nonstar} (ii), we also need to bound ${\mathbb{E}}|T_{n,\alpha}|^{2p+2}$.
Since $|X_{n,\alpha}|\le \alpha(n-1)$, we have $|\eta_{n,\alpha}|\le \sqrt{2\alpha}$ and 
therefore $|\eta_{n,\alpha}^{*}|\le \sqrt{2\alpha}$ (see (2.58) in \cite{CGS}).
Combining \eqref{c21} - \eqref{c23}, we have
\begin{equation}\label{c26}
\begin{split}
{\mathbb{E}}\left(|T_{n,\alpha}|^{2p+2}\right)
&\le \varepsilon^{2p+2}+ \left(8\alpha\right)^{p+1}{\mathbb{P}}\left(|T_{n,\alpha}|>\varepsilon\right)\\
&\le \varepsilon^{2p+2}+\dfrac{C_p \alpha^{p+1}}{n^{2e(p+2)-1}}\\
&\le C_p p^{2p}\left(\max\left\{\dfrac{1}{\sqrt{n}},\dfrac{\sqrt{\alpha}\log n}{n}\right\}\right)^{2p+2}\\
&\le C_p p^{2p}\left(\max\left\{\dfrac{1}{\sqrt{n}},\dfrac{\sqrt{\alpha}\log n}{n}\right\}\right)^{2}.
\end{split}
\end{equation}
Apply \thmref{nonstar} (ii), \eqref{n02} follows from \eqref{F19} and \eqref{c22}-\eqref{c26}.

To prove \eqref{n03}, we will need the following lemma.
\begin{lem}\label{g}
If there exist $\delta>0$ such that $\alpha\ge n^{1+\delta}$, then for all $p\ge0$, we have
\begin{equation}\label{g01}
\begin{split}
{\mathbb{E}}\left(|T_{n,\alpha}|^{p}\right)
&\le \dfrac{C_p p^{p}}{\delta^{p}} \left(\dfrac{\sqrt{\alpha}}{n}\right)^p.
\end{split}
\end{equation}
\end{lem}
\begin{proof}[Proof of \lemref{g}]
Let $L=1/(p+2)$ and let $q,\varepsilon'$ be as in \eqref{b12}. 
Then
\begin{equation}\label{d22}
\begin{split}
q\ge(p+2)^2\dfrac{n^\delta}{\delta^2}>8 \text{ and } \varepsilon' & \le  \dfrac{10(p+2)\sqrt{\alpha}}{\delta n}.
\end{split}
\end{equation}
From \eqref{b13}, we have
\begin{equation}\label{d23}
\begin{split}
{\mathbb{P}}_{\alpha}\left(|T_{n,\alpha}|>\varepsilon'\right)
&\le \dfrac{q}{\pi (q-1)n^{e(p+2)-1}}+\dfrac{eq^2(q-1)+q^2(q+1)}{\pi (q-1)^3(n-1)}\dfrac{\alpha}{n^{e(p+2)-1}}\\
&\le \dfrac{\alpha}{n^{2p+4}}.
\end{split}
\end{equation}
Similar to \eqref{c26}, \eqref{d23} yields
\begin{equation}\label{d26}
\begin{split}
{\mathbb{E}}\left(|T_{n,\alpha}|^{p}\right)
&\le (\varepsilon')^{p}+ \left(2\sqrt{2\alpha}\right)^{p}{\mathbb{P}}\left(|T_{n,\alpha}|>\varepsilon'\right)\\
&\le (\varepsilon')^{p}+\dfrac{C_p \alpha^{p/2+1}}{n^{2p+4}}\\
&\le \dfrac{C_p p^{p}}{\delta^{p}} \left(\dfrac{\sqrt{\alpha}}{n}\right)^p.
\end{split}
\end{equation}
The proof of \lemref{g} is completed.
\end{proof}
Now, we will prove \eqref{n03}. Let $L,\ q,\ \varepsilon'$ be as in the proof of \lemref{g}. From
\lemref{g}, we have
\begin{equation}\label{d27}
\begin{split}
{\mathbb{E}}\left(|T_{n,\alpha}|^{2p+2}\right)
&\le \dfrac{C_p p^{2p}}{\delta^{2p+2}} \left(\dfrac{\sqrt{\alpha}}{n}\right)^{2p+2}\le \dfrac{C_p p^{2p}}{\delta^{2p+2}} \left(\dfrac{\sqrt{\alpha}}{n}\right)^{2}.
\end{split}
\end{equation}
Apply \thmref{nonstar} (ii) with $\varepsilon'$ plays the role of $\varepsilon$ in \thmref{nonstar} (ii), \eqref{n03} follows from \eqref{F19}, \eqref{d23}, \eqref{d27} and the second half of \eqref{d22}.
\end{proof}
\begin{proof}[Proofs of \thmref{Ja1} and \thmref{Ja2}]
When $\alpha\ge 1$, \thmref{Ja1} is a direct consequence of \propref{Jack}. 
We also see that
\eqref{J11} holds if we replace $W_{n,\alpha}$ by $-W_{n,\alpha}$. To obtain \thmref{Ja1} for $0<\alpha<1$, we note
that from the definition of Jack measure, ${\mathbb{P}}_{\alpha}(\lambda)={\mathbb{P}}_{1/\alpha}(\lambda^t)$, where $\lambda^t$ is the transpose
 partition of $\lambda$. It also follows from \eqref{Fu01} and the definition of $\alpha$-content that
 $W_{n,\alpha}(\lambda)=-W_{n,1/\alpha}(\lambda^t)$. Therefore
\begin{align*}
{\mathbb{P}}_{\alpha}(W_{n,\alpha}=x)&={\mathbb{P}}_{\alpha}\left\{\lambda:W_{n,\alpha}(\lambda)=x\right\}\\
&={\mathbb{P}}_{1/\alpha}\left\{\lambda^t:W_{n,1/\alpha}(\lambda^t)=-x\right\}\\
&={\mathbb{P}}_{1/\alpha}\left(W_{n,1/\alpha}=-x\right).
\end{align*}
From this we
conclude that
${\mathbb{P}}_{\alpha}\left(W_{n,\alpha}\le x\right)={\mathbb{P}}_{1/\alpha}\left(W_{n,1/\alpha}\ge -x\right)$.
Therefore,
\begin{align*}
\sup_{x\in \mathbb{R}}|{\mathbb{P}}_{\alpha}(W_{n,\alpha}\le  x)-\Phi(x)|&=\sup_{x\in \mathbb{R}}\left|{\mathbb{P}}_{1/\alpha}\left(W_{n,1/\alpha}\ge  -x\right)-\Phi(x)\right|\\
&=\sup_{x\in \mathbb{R}}\left|{\mathbb{P}}_{1/\alpha}\left(-W_{n,1/\alpha}\le  x\right)-\Phi(x)\right|\\
&\le 8.2\max\left\{\dfrac{1}{\sqrt{n}},\dfrac{\log n}{\sqrt{\alpha}n}\right\}.
\end{align*}
Therefore, \thmref{Ja1} also holds when $0<\alpha<1$. This completes the proof of \thmref{Ja1}.

 When $1\le \alpha<n^2$, \thmref{Ja2} is a direct consequence of \propref{nonuni-Jack}. When $1/n^2<\alpha<1$, the proof is similar
 to that of \thmref{Ja1}, and this completes the proof of \thmref{Ja2}.
\end{proof}

\appendix
\section{}\label{sec:Appendix}
In this Section we will prove \eqref{rem_add_01} and two lemmas which are
used in Section \ref{sec:Rosenthal}.
\begin{proof}[Proof of \eqref{rem_add_01}]
For $p\ge 2$, applying \propref{ros} and \lemref{g} with noting that $\alpha\ge n^2$ (so that $\delta\ge 1$), we have

\begin{equation}\label{ap01}
\begin{split}
{\mathbb{E}}\left|W_{n,\alpha}\right|^p&\le \kappa_{p}\left(1+{\mathbb{E}}|T_{n,\alpha}|^{p-2}\right)\\
&\le \kappa_{p}\left(1+C_p p^p\left(\dfrac{\sqrt{\alpha}}{n}\right)^{p-2}\right)\\
&\le C_p \left(\dfrac{p^2}{\log p}\right)^p\left(\dfrac{\sqrt{\alpha}}{n}\right)^{p-2}
\end{split}
\end{equation}
establishing \eqref{rem_add_01}.
\end{proof}

\begin{lem}\label{lem_ap02}
	Let $p>8$ and let $\kappa_p$ be as in \propref{ros}, then
	\begin{equation}\label{a05_ap}
	\begin{split}
	\dfrac{\kappa_p}{\kappa_{p-2}}
	&\ge 8\left(\dfrac{p-1}{\log (p-1)}\right)^2.
	\end{split}
	\end{equation}
\end{lem}
\begin{proof}
	Let 
	\[h(p)=\log (\kappa_p)=p\left(\log p-\log (\log p)+\log (7/4)\right),\]
	we have
	\begin{equation}\label{a01}
	\begin{split}
	&h(p)-h(p-2)=\int_{p-2}^p h'(t)\mathrm{d}t\\
	&=2+2\log (7/4)+\int_{p-2}^p \log t \mathrm{d}t-\int_{p-2}^p \log (\log t) \mathrm{d}t- \int_{p-2}^p\dfrac{\mathrm{d}t}{\log t}.
	\end{split}
	\end{equation}
	Since the function $t\to \log(\log t)$ is concave,
	\begin{equation}\label{a02}
	\int_{p-2}^p \log (\log t) \mathrm{d}t \le 2 \log(\log(p-1)).
	\end{equation}
	Next, since $p>8$, we have 
	\begin{equation}\label{a03}
	\int_{p-2}^p \dfrac{\mathrm{d}t}{\log t} \le  \int_{6}^{8} \dfrac{\mathrm{d}t}{\log t},
	\end{equation}
	and
	\begin{equation}\label{a04}
	\begin{split}
	\int_{p-2}^p \log t \mathrm{d}t&=\int_{-1}^1 \log (p-1+s) \mathrm{d}s\\
	&=2\log(p-1)+\int_{-1}^1\log\left(1+\dfrac{s}{p-1}\right) \mathrm{d}s\\
	&=2\log(p-1)+\int_{0}^1\log\left(1-\dfrac{s^2}{(p-1)^2}\right) \mathrm{d}s\\
	&\ge 2\log(p-1)+\int_{0}^1\log\left(1-\dfrac{s^2}{49}\right) \mathrm{d}s.
	\end{split}
	\end{equation}
	Combining \eqref{a01}-\eqref{a04}, numerical calculation gives
	\begin{equation*}
	\begin{split}
	h(p)-h(p-2)&\ge 2+2\log (7/4) +2\log (p-1)- 2\log (\log (p-1))\\
	&\quad\quad -\int_{6}^8\dfrac{\mathrm{d}t}{\log t}+\int_{0}^1\log\left(1-\dfrac{s^2}{49}\right)\mathrm{d}s\\
	&>2\log (p-1)- 2\log (\log (p-1))+\log 8
	\end{split}
	\end{equation*}
	for all $p>8$. This implies \eqref{a05_ap}.
\end{proof}

\begin{lem}\label{lem_ap01} Let $p>8$ and let $\kappa_p$ be as in \propref{ros}. Then
\begin{equation}\label{ap03}
\begin{split}
\dfrac{2(p-1)}{\left(1-\theta\right)^{p-3}}\le \kappa_{p},
\end{split}
\end{equation}
where
\[\theta=\theta(p):=\left(\dfrac{\log^2(p-1)}{4(p-1)}\right)^{1/(p-3)}.\]
\end{lem}
\begin{proof}
Firstly, we will prove that
\begin{equation}\label{ap04}
\begin{split}
\dfrac{1}{1-\theta}\le \dfrac{7p}{4\log p}
\end{split}
\end{equation}
which is equivalent to 
\begin{equation}\label{ap05}
\begin{split}
\dfrac{1}{p-3}\log\left(\dfrac{\log^2(p-1)}{4(p-1)}\right)\le \log\left(1-\dfrac{4\log p}{7p}\right).
\end{split}
\end{equation}
Since $\dfrac{4\log p}{7p}$ is decreasing when $p>8$,
\begin{equation}\label{ap06}
\dfrac{4\log p}{7p}\le 0.1486.
\end{equation}
On the other hand, it is easy to prove that
\begin{equation}\label{ap07}
\log(1-x)\ge \dfrac{-13x}{12} \text{ for all } 0\le x\le 0.1486.
\end{equation}
From \eqref{ap06} and \eqref{ap07}, we have
\begin{equation}\label{ap09}
\begin{split}
\log\left(1-\dfrac{4\log p}{7p}\right) \ge -\dfrac{13\log p}{21p}. 
\end{split}
\end{equation}
Therefore, to prove \eqref{ap05}, it suffices to prove that
\begin{equation}\label{ap12}
\begin{split}
\dfrac{1}{p-3}\log\left(\dfrac{\log^2(p-1)}{4(p-1)}\right)\le  -\dfrac{13\log p}{21p}
\end{split}
\end{equation}
which is equivalent to 
\begin{equation}\label{ap18}
R_1(p)+R_2(p)\ge 0, 
\end{equation}
where
\[R_1(p)=\dfrac{13\log(p-1)}{21(p-3)}-\dfrac{13\log p}{21p},\]
and
\[R_2(p)=\dfrac{(8/21)\log(p-1)-2\log \left(\log(p-1)\right)+\log(4)}{p-3}.\]
Elementary calculus shows that $R_1(p)\ge 0$ and $R_2(p)\ge 0$ for all $p>8$. Therefore \eqref{ap12} holds, completing the proof of \eqref{ap04}.

Now, we will prove \eqref{ap03}. Since $p>8$, we have from \eqref{ap04} that
\begin{equation*}
\begin{split}
\dfrac{2(p-1)}{\left(1-\theta\right)^{p-3}}&\le 2(p-1)\left(\dfrac{4\log p}{7p}\right)^{3}\left(\dfrac{7p}{4\log p}\right)^{p}\le \dfrac{\left(\log 8\right)^3}{196} \left(\dfrac{7p}{4\log p}\right)^{p}=\kappa_p.
\end{split}
\end{equation*}
The proof of the lemma is completed.

\end{proof}

\section*{Acknowledgements}
The first author and the third author were partially supported by Grant R-146-000-182-112 and Grant R-146-000-230-114 from the
National University of Singapore.
A substantial part of this paper was written when the third author was at the Institute for Mathematical Sciences (IMS) and Department of Mathematics, National University of Singapore. He would like to thank the IMS staff and the Department of Mathematics for their hospitality.

\end{document}